\documentclass{article}
\usepackage{amsfonts}
\usepackage{amsmath}
\usepackage{graphicx}

\newtheorem{theorem}{Theorem}

\newtheorem{example}[theorem]{Example}

\newtheorem{lemma}[theorem]{Lemma}

\newtheorem{proposition}[theorem]{Proposition}

\newenvironment{proof}[1][Proof]{\noindent\textbf{#1.} }{\ \rule{0.5em}{0.5em}}
\input{tcilatex}
\begin{document}

\title{Optimal control of discrete-time linear fractional order systems with
multiplicative noise }
\author{J.J. Trujillo\ \thanks{%
Departamento de An\'{a}lisis Matem\'{a}tico, Universidad de La Laquna, 38271
La Laguna, e-mail jtrujill@ullmat.es}, V.M. Ungureanu \thanks{%
Department of Mathematics, \textquotedblleft Constantin
Brancusi\textquotedblright\ University, Tg. Jiu, Bulevardul Republicii, nr.
1, jud. Gorj, Romaina, e-mail vio@utgjiu.ro.}}
\date{}
\maketitle

\begin{abstract}
A finite horizon linear quadratic(LQ) optimal control problem is studied for
a class of discrete-time linear fractional systems (LFSs) affected by
multiplicative, independent random perturbations. Based on the dynamic
programming technique, two methods are proposed for solving this problem.
The first one seems to be new and uses a linear, expanded-state model of the
LFS. The LQ optimal control problem reduces to a similar one for stochastic
linear systems and the solution is obtained by solving Riccati equations.
The second method appeals to the Principle of Optimality and provides an
algorithm for the computation of the optimal control and cost by using
directly the fractional system. As expected, in both cases the optimal
control is a linear function in the state and can be computed by a computer
program. Two numerical examples proves the effectiveness of each method.
\end{abstract}

\section{Introduction}

Fractional calculus(FC) began to engage mathematicians' interest in the 17th
century as evidenced by a letter of Leibniz to L'H\^{o}spital, dated 30$%
^{th} $ September 1695, which talks about the possibility of non-integer
order differentiation. Later on, famous mathematicians as Fourier, Euler and
Laplace contributed to the foundation of this new branch of mathematics with
various concepts and results. Nowadays, the most popular definitions of the
non-integer order integral or derivative are the Riemann-Liouville, Caputo
and Grunwald-Letnikov definitions. For a historical survey and the current
state of the art, the reader is referred to \cite{Mi}, \cite{Woj}, \cite{Ka},%
\cite{Concept}, \cite{stab} and the references therein.

FC finds use in different fields of science and engineering including the
electrochemistry, electromagnetism, biophysics, quantum mechanics, radiation
physics, statistics or control theory (see \cite{Concept}, \cite{Ka}, \cite%
{Mi}). Such an example comes from the field of autonomous guided vehicles,
which lateral control seems to be improved by using fractional adaptation
schemes \cite{Vin}. Also, partial differential equations of fractional order
were applied to model the wave propagation in viscoelastic media or the
dissipation in seismology or in metallurgy \cite{Mainardi}.

The optimal control theory was intensively developed during the last century
for deterministic systems defined by integer-order derivatives, in both
continuous- and discrete- time cases \cite{Moore}. Since many real-world
phenomena are affected by random factors that exercised a decisive influence
on the processes behavior, stochastic optimal control theory had a similar
evolution in the recent decades (see \cite{atha}, \cite{Cdrag}, \cite{1995}, 
\cite{Optim} and the references therein).

However, only a few papers address optimal control problems for fractional
systems (see e.g. \cite{Tric}, \cite{Agr}, \cite{Agr2}, \cite{Idiri}, \cite%
{kamo}, \cite{polonezi}) and fewer consider stochastic fractional systems 
\cite{SA}, \cite{Amir}.

In this paper we formulate a finite-horizon LQ optimal control problem for
stochastic discrete-time LFSs defined by the Grunwald-Letnikov fractional
derivative. As far as we know, this subject seems to be new.

We use the classical dynamic programming technique to derive two methods for
solving the proposed optimal control problem. Obviously these methods apply
to deterministic discrete-time LFSs.

The first one is new and uses an equivalent linear \textit{expanded-state }%
model of the stochastic LFS. As the name says, the state of this model is
expanded and formed by the actual state and all the past states of the LFS 
\cite{Concept}. The quadratic cost functional is rewritten accordingly and
the original optimal control problem reduces to a LQ optimal control problem
for linear stochastic systems. Since the control weight matrix of the new
optimization problem is not positive, we modify it with a parameter $%
\varepsilon >0$ for achieving the positivity condition. This perturbation is
chosen such that the optimal value of the new performance index (denoted $%
I_{x_{0},N,\varepsilon }(U)$) is independent of $\varepsilon $ and coincides
with the optimal value of the original cost functional ($I_{x_{0},N}(u)$)
(see Proposition \ref{minimul nu dep de eps}). The optimal value of the
performance index $I_{x_{0},N,\varepsilon }(U)$ is a quadratic form in the
initial expanded-state and can be computed by solving a classical matrix
Riccati equation. The optimal feedback law $U$ is linear in state, and
involves the solution of the same Riccati equation. The optimal control
sequence $u$ of the original optimal control problem is computable from $U$.

The second method is a stochastic version of the new algorithm proposed in 
\cite{polonezi} for deterministic LFSs. It uses the \textit{Optimality
Principle} for computing recursively (and starting with the terminal time)
the optimal control sequence $u_{N-1},..,u_{0}$ and the optimal cost.

The main difference between the two methods is that the dynamic programing
approach is applied in the first case to a linear stochastic system, while,
in the second case, the same technique is applied directly to a stochastic
LFS.

To compare the two methods, a numerical example is solved by using two
computer algorithms written for this purpose. As expected, the mathematical
results are the same, but the run-time of the program that implements the
algorithm provided by the first method seems to be shorter.

The paper is organized as follows. In Section 2, we shortly review necessary
notions from FC and we state the optimal control problem $\mathcal{O}$. In
Section 3 we reformulate the problem $\mathcal{O}$ by using the equivalent
linear \textit{expanded-state }model of the stochastic LFS and a
parametrized cost functional $I_{x_{0},N,\varepsilon }(U)$. As mentioned
above, the optimal control and cost can be computed with the solution of an
associated Riccati equation. Finally, a numerical example illustrates the
effectiveness of this first method.

In Section 4 we present the first two steps of the recursive algorithm
(called Algorithm $A$) which starts with the terminal time and provides the
optimal control sequence and cost that solves problem $\mathcal{O}$. The
general step of the Algorithm $A$ is described in the Appendix. The
numerical example presented at the end in Section 4 is solved by using
Algorithm $A$. It proves the applicability of the second method. Some
conclusions are drawn in the last section.

\section{Notations and statement of the problem}

As usual, $%
%TCIMACRO{\U{211d} }%
%BeginExpansion
\mathbb{R}
%EndExpansion
$ is the set of real numbers, $%
%TCIMACRO{\U{211d} }%
%BeginExpansion
\mathbb{R}
%EndExpansion
^{d},d\in 
%TCIMACRO{\U{2115} }%
%BeginExpansion
\mathbb{N}
%EndExpansion
^{\ast }=%
%TCIMACRO{\U{2115} }%
%BeginExpansion
\mathbb{N}
%EndExpansion
-\{0\}$ is the real Hilbert space of real $d$-dimensional vectors and $%
%TCIMACRO{\U{211d} }%
%BeginExpansion
\mathbb{R}
%EndExpansion
^{d\times n}$, $n\in 
%TCIMACRO{\U{2115} }%
%BeginExpansion
\mathbb{N}
%EndExpansion
^{\ast }$ is the linear space of $d\times n$ real matrices. We also denote
by $\left( 
%TCIMACRO{\U{211d} }%
%BeginExpansion
\mathbb{R}
%EndExpansion
^{d}\right) ^{n}$ the Hilbert space of all $n$ dimensional vectors from $%
%TCIMACRO{\U{211d} }%
%BeginExpansion
\mathbb{R}
%EndExpansion
^{d}$. Obviously it is isomorphic with $%
%TCIMACRO{\U{211d} }%
%BeginExpansion
\mathbb{R}
%EndExpansion
^{d\times n}$. In this paper we do not distinguish between a linear operator
on $%
%TCIMACRO{\U{211d} }%
%BeginExpansion
\mathbb{R}
%EndExpansion
^{d\times n}$ (or $%
%TCIMACRO{\U{211d} }%
%BeginExpansion
\mathbb{R}
%EndExpansion
^{d}$) and the associated matrix. Also we shall write\ $\left\langle
.,.\right\rangle $ for the inner product and $\left\Vert .\right\Vert $ for
norms of elements and operators. For any linear operator $T$ acting on
finite dimensional real spaces, we denote by $T^{\ast }$ the adjoint
operator of $T$. We say that $T:%
%TCIMACRO{\U{211d} }%
%BeginExpansion
\mathbb{R}
%EndExpansion
^{n}\rightarrow 
%TCIMACRO{\U{211d} }%
%BeginExpansion
\mathbb{R}
%EndExpansion
^{n}$ is nonnegative (we write $T\geq 0$) if $T=T^{\ast }$ and $\left\langle
Tx,x\right\rangle \geq 0,$ for all $x\in 
%TCIMACRO{\U{211d} }%
%BeginExpansion
\mathbb{R}
%EndExpansion
^{n}$; $T$ is positive (we write $T>0$) if $T\geq 0$ and there is $\delta >0$
such that $\left\langle Tx,x\right\rangle \geq \delta \left\Vert
x\right\Vert ^{2},$ for all $x\in 
%TCIMACRO{\U{211d} }%
%BeginExpansion
\mathbb{R}
%EndExpansion
^{n}$. The identity operator on $%
%TCIMACRO{\U{211d} }%
%BeginExpansion
\mathbb{R}
%EndExpansion
^{n}$ will be denoted by $I_{%
%TCIMACRO{\U{211d} }%
%BeginExpansion
\mathbb{R}
%EndExpansion
^{d}}$.

Let $\alpha \in \left( 0,2\right) $ and $h>0$ be fixed. We recall that for
all $j\in \mathbf{%
%TCIMACRO{\U{2115} }%
%BeginExpansion
\mathbb{N}
%EndExpansion
}$, $\left( 
\begin{array}{c}
\alpha \\ 
j%
\end{array}%
\right) $ denotes the generalized binomial coefficient defined by 
\begin{equation*}
\left( 
\begin{array}{c}
\alpha \\ 
j%
\end{array}%
\right) =\left\{ 
\begin{array}{c}
1,j=0 \\ 
\frac{\alpha \left( \alpha -1\right) \cdot ...\cdot \left( \alpha
+1-j\right) }{j!},j\in \mathbf{%
%TCIMACRO{\U{2115} }%
%BeginExpansion
\mathbb{N}
%EndExpansion
}^{\ast }%
\end{array}%
\right. .
\end{equation*}%
Then, for any sequence $\{x_{k}\}_{k\in 
%TCIMACRO{\U{2115} }%
%BeginExpansion
\mathbb{N}
%EndExpansion
}\subset 
%TCIMACRO{\U{211d} }%
%BeginExpansion
\mathbb{R}
%EndExpansion
^{d},d\in 
%TCIMACRO{\U{2115} }%
%BeginExpansion
\mathbb{N}
%EndExpansion
$%
\begin{equation*}
\Delta ^{\left[ \alpha \right] }x_{k+1}=\frac{1}{h^{\alpha }}%
\dsum\limits_{j=0}^{k+1}\left( -1\right) ^{j}\left( 
\begin{array}{c}
\alpha \\ 
j%
\end{array}%
\right) x_{k+1-j},h>0
\end{equation*}%
is the discrete fractional-order operator that arises in the Gr\"{u}%
nwald-Letnikov definition of the fractional order derivatives (see for e.g. 
\cite{Concept}).

Let $\{\xi _{k}\}_{k\in \mathbf{%
%TCIMACRO{\U{2115} }%
%BeginExpansion
\mathbb{N}
%EndExpansion
}}$ be a sequence of real-valued, mutually independent random variables on
the probability space $\left( \Omega ,\mathcal{G},P\right) $ that satisfies
the condition $E\left[ \xi _{k}\right] =0,E\left[ \xi _{k}^{2}\right]
=1,k\in \mathbf{%
%TCIMACRO{\U{2115} }%
%BeginExpansion
\mathbb{N}
%EndExpansion
}$\textbf{. (} Here $E\left[ \xi \right] $ is the mean (expectation) of $\xi
_{k}$.) The $\sigma -$ algebra generated by $\{\xi _{i},0\leq i\leq
n-1\},n\in \mathbf{%
%TCIMACRO{\U{2115} }%
%BeginExpansion
\mathbb{N}
%EndExpansion
}^{\ast }$ will be denoted by $\mathcal{G}_{n}$. We consider the stochastic
discrete-time fractional system with control%
\begin{eqnarray}
\Delta ^{\left[ \alpha \right] }x_{k+1} &=&\mathbb{A}x_{k}+\xi _{k}\mathbb{B}%
x_{k}+\mathbb{D}u_{k}+\xi _{k}\mathbb{F}u_{k},k\in 
%TCIMACRO{\U{2115} }%
%BeginExpansion
\mathbb{N}
%EndExpansion
\label{ec 1} \\
x_{0} &=&x\in 
%TCIMACRO{\U{211d} }%
%BeginExpansion
\mathbb{R}
%EndExpansion
^{d},  \notag
\end{eqnarray}%
where $\mathbb{A}$, $\mathbb{B\in }$ $%
%TCIMACRO{\U{211d} }%
%BeginExpansion
\mathbb{R}
%EndExpansion
^{d\times d}$, $\mathbb{D}$, $\mathbb{F}$ $\mathbb{\in }$ $%
%TCIMACRO{\U{211d} }%
%BeginExpansion
\mathbb{R}
%EndExpansion
^{d\times m},m\in 
%TCIMACRO{\U{2115} }%
%BeginExpansion
\mathbb{N}
%EndExpansion
$ and the control $u=\{u_{k}\}_{k\in 
%TCIMACRO{\U{2115} }%
%BeginExpansion
\mathbb{N}
%EndExpansion
}$ belongs to a class of admissible controls $\mathcal{U}^{a}$ formed by all
sequences $u$ which elements $u_{k}$ are $\mathcal{G}_{k}$-measurable, $%
%TCIMACRO{\U{211d} }%
%BeginExpansion
\mathbb{R}
%EndExpansion
^{m}$-valued random variables satisfying $E\left[ \left\Vert
u_{k}\right\Vert ^{2}\right] <\infty $ for all $k\in 
%TCIMACRO{\U{2115} }%
%BeginExpansion
\mathbb{N}
%EndExpansion
$.

A finite segment of an admissible control sequence $u$ is of the form $%
u_{k},u_{k+1}$, ... , $u_{N}$. In the sequel we shall denote by $\mathcal{U}%
_{k,N-1}^{a}$ the set of segments $u_{k},u_{k+1},...$, $u_{N-1}$ of
admissible controls $u\in \mathcal{U}^{a}.$

Multiplying (\ref{ec 1}) by $h^{\alpha }$ and denoting $A_{0}=h^{\alpha }%
\mathbb{A}+\alpha I_{%
%TCIMACRO{\U{211d} }%
%BeginExpansion
\mathbb{R}
%EndExpansion
^{d}}$, $T=h^{\alpha }\mathbb{T}$, for any $T=B$, $D$, $F$, $\mathbb{T}=%
\mathbb{B}$, $\mathbb{D}$, $\mathbb{F}$, $c_{j}:=\left( -1\right) ^{j}\left( 
\begin{array}{c}
\alpha \\ 
j+1%
\end{array}%
\right) $ and $A_{j}=c_{j}I_{%
%TCIMACRO{\U{211d} }%
%BeginExpansion
\mathbb{R}
%EndExpansion
^{d}}$, system (\ref{ec 1}) can be equivalently rewritten as%
\begin{eqnarray}
x_{k+1} &=&\dsum\limits_{j=0}^{k}A_{j}x_{k-j}+\xi _{k}Bx_{k}+Du_{k}+\xi
_{k}Fu_{k},  \label{ec c2} \\
x_{0} &=&x\in 
%TCIMACRO{\U{211d} }%
%BeginExpansion
\mathbb{R}
%EndExpansion
^{d}.  \label{ci c2}
\end{eqnarray}

Let $x_{0}\in 
%TCIMACRO{\U{211d} }%
%BeginExpansion
\mathbb{R}
%EndExpansion
^{d}$ and $N\in 
%TCIMACRO{\U{2115} }%
%BeginExpansion
\mathbb{N}
%EndExpansion
$ be fixed, $C\in 
%TCIMACRO{\U{211d} }%
%BeginExpansion
\mathbb{R}
%EndExpansion
^{p\times d},S\in 
%TCIMACRO{\U{211d} }%
%BeginExpansion
\mathbb{R}
%EndExpansion
^{d\times d},S\geq 0$ and $K\in 
%TCIMACRO{\U{211d} }%
%BeginExpansion
\mathbb{R}
%EndExpansion
^{m\times m},K>0$.

Our \textit{optimal control problem }$\mathcal{O}$ is to minimize the cost
functional 
\begin{gather}
I_{x_{0},N}(u)=  \label{cost funct} \\
\sum\limits_{n=0}^{N-1}E\left[ \left( \left\Vert Cx_{n}\right\Vert
^{2}+<Ku_{n},u_{n}>\right) \right] +E<Sx_{N},x_{N}>  \notag
\end{gather}%
subject to (\ref{ec c2})-(\ref{ci c2}), over the class $\mathcal{U}%
_{0,N-1}^{a}$ of segments of admissible controls.

\section{An equivalent optimal control problem for a non-fractional linear
system}

In this section we first present an equivalent linear \textit{expanded-state 
}model (see (\ref{sl 1})-(\ref{ci sl1})) of the stochastic LFS. Then we show
that optimal control problem $\mathcal{O}$ is "equivalent" with a LQ optimal
control problem associated with (\ref{sl 1})-(\ref{ci sl1}). The word
"equivalent" means here that the two optimal control problems have the same
optimal costs and an optimal control sequence (OCS) of the one can be
obtained from an OCS of the other. Since the solution of the new optimal
control problem can be obtained by solving a backward discrete-time Riccati
equation, we get a solution of $\mathcal{O}$(see Theorem \ref{t opt}).

\subsection{A linear expanded-state\textit{\ }model}

Let $\mathcal{A},\mathcal{B}:\left( 
%TCIMACRO{\U{211d} }%
%BeginExpansion
\mathbb{R}
%EndExpansion
^{d}\right) ^{N}\rightarrow \left( 
%TCIMACRO{\U{211d} }%
%BeginExpansion
\mathbb{R}
%EndExpansion
^{d}\right) ^{N}$ be the linear operators defined by the matrices 
\begin{equation}
\mathcal{A}=\left( 
\begin{array}{cccc}
A_{0} & c_{1}I_{%
%TCIMACRO{\U{211d} }%
%BeginExpansion
\mathbb{R}
%EndExpansion
^{d}} & ... & c_{N-1}I_{%
%TCIMACRO{\U{211d} }%
%BeginExpansion
\mathbb{R}
%EndExpansion
^{d}} \\ 
I_{%
%TCIMACRO{\U{211d} }%
%BeginExpansion
\mathbb{R}
%EndExpansion
^{d}} & 0 & . & 0 \\ 
. & I_{%
%TCIMACRO{\U{211d} }%
%BeginExpansion
\mathbb{R}
%EndExpansion
^{d}} & . & . \\ 
. & . & I_{%
%TCIMACRO{\U{211d} }%
%BeginExpansion
\mathbb{R}
%EndExpansion
^{d}} & 0%
\end{array}%
\right) ,\mathcal{B}=\left( 
\begin{array}{ccccc}
B & 0 & . & . & . \\ 
0 & 0 & . & . & . \\ 
. & . & . & . & . \\ 
. & . & . & . & .%
\end{array}%
\right) .  \label{def Ak}
\end{equation}%
Also let $\mathcal{D}_{k},\mathcal{F}_{k}:(%
%TCIMACRO{\U{211d} }%
%BeginExpansion
\mathbb{R}
%EndExpansion
^{m})^{N}\rightarrow (%
%TCIMACRO{\U{211d} }%
%BeginExpansion
\mathbb{R}
%EndExpansion
^{d})^{N},k=0,..,N-1$ be given by 
\begin{equation*}
\mathcal{D}_{k}\left( v_{0},v_{1},...,v_{N-1}\right) =\left(
Dv_{k},0,...,0\right) \in (%
%TCIMACRO{\U{211d} }%
%BeginExpansion
\mathbb{R}
%EndExpansion
^{d})^{N}
\end{equation*}
and\ 
\begin{equation*}
\mathcal{F}_{k}\left( v_{0},v_{1},...,v_{N-1}\right) =\left(
Fv_{k},0,...,0\right) \in (%
%TCIMACRO{\U{211d} }%
%BeginExpansion
\mathbb{R}
%EndExpansion
^{d})^{N},
\end{equation*}
for all $\left( v_{0},v_{1},...,v_{N-1}\right) \in (%
%TCIMACRO{\U{211d} }%
%BeginExpansion
\mathbb{R}
%EndExpansion
^{m})^{N}$.

Similarly, for all $k=0,..,N-1$, we define $\mathcal{K}_{k}:(%
%TCIMACRO{\U{211d} }%
%BeginExpansion
\mathbb{R}
%EndExpansion
^{m})^{N}\rightarrow (%
%TCIMACRO{\U{211d} }%
%BeginExpansion
\mathbb{R}
%EndExpansion
^{m})^{N}$,$\mathcal{C}:(%
%TCIMACRO{\U{211d} }%
%BeginExpansion
\mathbb{R}
%EndExpansion
^{d})^{N}\rightarrow (%
%TCIMACRO{\U{211d} }%
%BeginExpansion
\mathbb{R}
%EndExpansion
^{p})^{N}$ 
\begin{eqnarray*}
\mathcal{K}_{k}\left( v_{0},v_{1},...,v_{N-1}\right) &=&\left(
0,...,Kv_{k},0,...,0\right) \in 
%TCIMACRO{\U{211d} }%
%BeginExpansion
\mathbb{R}
%EndExpansion
^{m\times N} \\
\mathcal{C}\left( v_{0},v_{1},...,v_{N}-1\right) &=&\left(
Cv_{0},0,...,0\right)
\end{eqnarray*}

and $\mathcal{S}:(%
%TCIMACRO{\U{211d} }%
%BeginExpansion
\mathbb{R}
%EndExpansion
^{d})^{N}\rightarrow $ $(%
%TCIMACRO{\U{211d} }%
%BeginExpansion
\mathbb{R}
%EndExpansion
^{d})^{N}$ 
\begin{equation*}
\mathcal{S}\left( v_{0},...,v_{N-1}\right) =\left( Sv_{0},0,...,0\right) .
\end{equation*}%
Obviously, $\mathcal{K}_{k},\mathcal{S}\geq 0$. Let $%
x_{0},x_{1},...,x_{k},...$ be a solution of (\ref{ec c2}). For any $k<N%
\mathbf{,}X_{k}^{T}=\left( x_{k},x_{k-1},...,x_{0},0,..,\underset{N}{0}%
\right) \in $ $\left( 
%TCIMACRO{\U{211d} }%
%BeginExpansion
\mathbb{R}
%EndExpansion
^{d}\right) ^{N}$ is a solution of the discrete-time system with independent
random perturbations%
\begin{eqnarray}
X_{k+1} &=&\mathcal{A}X_{k}+\xi _{k}\mathcal{B}X_{k}+\mathcal{D}%
_{k}U_{k}+\xi _{k}\mathcal{F}_{k}U_{k},  \label{sl 1} \\
X_{0} &=&\left( x_{0},0,...,\underset{N}{0}\right) ,  \label{ci sl1}
\end{eqnarray}%
where the control $U=\{U_{k}\}_{k\in 
%TCIMACRO{\U{2115} }%
%BeginExpansion
\mathbb{N}
%EndExpansion
}\subset \left( 
%TCIMACRO{\U{211d} }%
%BeginExpansion
\mathbb{R}
%EndExpansion
^{m}\right) ^{N}$ belongs to the set $\mathbb{U}^{a}$ of admissible controls
sequences $\{U_{k}\}_{k\in 
%TCIMACRO{\U{2115} }%
%BeginExpansion
\mathbb{N}
%EndExpansion
}$ having the property that $U_{k}$ are $\left( 
%TCIMACRO{\U{211d} }%
%BeginExpansion
\mathbb{R}
%EndExpansion
^{m}\right) ^{N}$-valued, $\mathcal{G}_{k}$-measurable random variables
satisfying $E\left[ \left\Vert U_{k}\right\Vert ^{2}\right] <\infty $ for
all $k\in 
%TCIMACRO{\U{2115} }%
%BeginExpansion
\mathbb{N}
%EndExpansion
$. The system (\ref{sl 1})-(\ref{ci sl1}) is a classical linear
discrete-time control system with independent random perturbations. We know
(see, e.g \cite{Cdrag}) that for all $k\in 
%TCIMACRO{\U{2115} }%
%BeginExpansion
\mathbb{N}
%EndExpansion
^{\ast }$, $X_{k}$ is $\mathcal{G}_{k}$-measurable and the pair $X_{k},\xi
_{n}$ is independent for all $n\geq k>0$.

Computing $X_{N}$ from (\ref{sl 1})-(\ref{ci sl1}),we note that $%
X_{N}=\left( x_{N},x_{N-1},...,x_{1}\right) $ and $x_{N}$, the $n$-th
solution of (\ref{ec c2})-(\ref{ci c2}), is the first component of $X_{N}$.
Then $\left\langle \mathcal{S}X_{N},X_{N}\right\rangle =\left\langle
Sx_{n},x_{n}\right\rangle .$ Also, for all $k<N$ we have 
\begin{eqnarray*}
\mathcal{C}X_{k} &=&\mathcal{C}\left( x_{k},x_{k-1},...,x_{0},0,..0\right)
=\left( Cx_{k},0,...,0\right) , \\
\mathcal{K}_{k}U_{k} &=&\mathcal{K}_{k}\left( \overline{u}_{0},\overline{u}%
_{1},...,u_{k}..,\overline{u}_{N-1}\right) =\left( 0,0,...,\underset{k+1}{%
Ku_{k}},...,0\right) .
\end{eqnarray*}

Then, the cost functional (\ref{cost funct}) can be equivalently rewritten
as 
\begin{equation}
I_{x_{0},N}(U)=E\left[ \sum\limits_{k=0}^{N-1}\left\langle \mathcal{C}^{\ast
}\mathcal{C}X_{k},X_{k}\right\rangle +\left\langle \mathcal{S}%
X_{N},X_{N}\right\rangle +\left\langle \mathcal{K}_{k}U_{k},U_{k}\right%
\rangle \right] .  \label{cost mod}
\end{equation}

Substituting $X_{N}$ given by (\ref{sl 1})-(\ref{ci sl1}) in (\ref{cost mod}%
), we get%
\begin{eqnarray}
I_{x_{0},N}(U) &=&\sum\limits_{n=0}^{N-2}E[\left\langle \mathcal{C}^{\ast }%
\mathcal{C}X_{k},X_{k}\right\rangle +\left\langle \mathcal{K}%
_{k}U_{k},U_{k}\right\rangle ]  \notag \\
&&+E[\left\langle \left( \mathcal{C}^{\ast }\mathcal{C}+\mathcal{A}^{\ast }%
\mathcal{SA+B}^{\ast }\mathcal{SB}\right) X_{N-1},X_{N-1}\right\rangle 
\notag \\
&&+2\left\langle \left( \mathcal{D}_{N-1}^{\ast }\mathcal{SA+F}_{N-1}^{\ast }%
\mathcal{SB}\right) X_{N-1},U_{N-1}\right\rangle +  \label{for cost} \\
&&\left\langle \left( \mathcal{K}_{N-1}+\mathcal{D}_{N-1}^{\ast }\mathcal{SD}%
_{N-1}\mathcal{+F}_{N-1}^{\ast }\mathcal{SF}_{N-1}\right)
U_{N-1},U_{N-1}\right\rangle ].  \notag
\end{eqnarray}

To obtain the last equality we have applied the property of $X_{N-1}$ and $%
U_{N-1}$ of being independent of $\xi _{N-1}$. Thus for any appropriate
deterministic linear operators $V$ and $T,$ we have 
\begin{gather*}
E\left[ \left\langle VX_{N-1}+\xi _{N-1}TU_{N-1},VX_{N-1}+\xi
_{N-1}TU_{N-1}\right\rangle \right] = \\
E\left[ \left\langle VX_{N-1},VX_{N-1}\right\rangle \right] +2E\left[ \xi
_{N-1}\right] E\left[ \left\langle VX_{N-1},TU_{N-1}\right\rangle \right] +
\\
E\left[ \xi _{N-1}^{2}\right] E\left[ \left\langle
TU_{N-1},TU_{N-1}\right\rangle \right] \\
=E\left[ \left\langle VX_{N-1},VX_{N-1}\right\rangle \right] +E\left[
\left\langle TU_{N-1},TU_{N-1}\right\rangle \right]
\end{gather*}%
and (\ref{for cost}) follows.

Now let $\mathbb{U}_{0,N-1}^{a}$ be the class of all finite segments $%
U_{0},...,U_{N-1}$ of sequences $U\in \mathbb{U}^{a}$. It is not difficult
to see that the optimal control problem $\mathcal{O}$ is equivalent with the
minimizing optimal control problem $\mathcal{O}_{1}$ defined by system (\ref%
{sl 1})-(\ref{ci sl1}), $I_{x_{0},N}(U)$ and $\mathbb{U}_{0,N-1}^{a}$.
Indeed, for any $u\in \mathcal{U}_{0,N-1}^{a}$, the segment $%
U=\{U_{k}=\left( 0,..,\underset{k+1}{u_{k}},...0\right) ,k=0,..,N-1\}$
belongs to $\mathbb{U}_{0,N-1}^{a}$ and $I_{x_{0},N}(U)=I_{x_{0},N}(u)$.
Conversely, given $U\in \mathbb{U}_{0,N-1}^{a},$ we define $%
u=\{u_{k}=U_{kk},k=0,..,N-1\}$. Thus, $u\in \mathcal{U}_{0,N-1}^{a}$ and $%
I_{x_{0},N}(U)=I_{x_{0},N}(u)$. Now it is clear that $\widetilde{U}$ is
optimal for $I_{x_{0},N}(U)$ if and only if $\widetilde{u}=\{\widetilde{u}%
_{k}=\widetilde{U}_{kk},k=0,..,N-1\}$ is optimal for $I_{x_{0},N}(u)$ and $%
I_{x_{0},N}(\widetilde{U})=I_{x_{0},N}(\widetilde{u})$.

The problem $\mathcal{O}_{1}$ is a linear quadratic optimal control problem
for stochastic systems. However $\mathcal{K}_{k}$ does not satisfy the
condition $\mathcal{K}_{k}>0$, $k=0,..,N-1$ and we cannot solve $\mathcal{O}%
_{1}$ by a direct application of the known results from the optimal control
theory of stochastic discrete-time systems (see \cite{Moore}, \cite{Cdrag}).

Therefore, we replace the optimal cost $I_{x_{0},N}(U)$ from $\mathcal{O}%
_{1} $ with the optimal cost 
\begin{eqnarray}
I_{x_{0},N,\varepsilon }(U) &=&\sum\limits_{k=0}^{N-2}E\left[ \left\langle 
\mathcal{C}^{\ast }\mathcal{C}X_{k},X_{k}\right\rangle \right] +E\left[
\left\langle \left( \mathcal{K}_{k}+\varepsilon \mathcal{I}_{k}\right)
U_{k},U_{k}\right\rangle \right]  \label{cost eps} \\
&&+E[\left\langle \left( \mathcal{C}^{\ast }\mathcal{C}+\mathcal{A}^{\ast }%
\mathcal{SA+B}^{\ast }\mathcal{SB}\right) X_{N-1},X_{N-1}\right\rangle 
\notag \\
&&+2\left\langle \left( \mathcal{D}_{N-1}^{\ast }\mathcal{SA+F}_{N-1}^{\ast }%
\mathcal{SB}\right) X_{N-1},U_{N-1}\right\rangle  \notag \\
&&+\left\langle \mathbb{K}_{N-1}^{\varepsilon }U_{N-1},U_{N-1}\right\rangle ]
\notag
\end{eqnarray}%
where $\varepsilon >0$ is fixed, 
\begin{equation}
\mathbb{K}_{N-1}^{\varepsilon }=\mathcal{K}_{N-1}+\varepsilon \mathcal{I}%
_{N-1}+\mathcal{D}_{N-1}^{\ast }\mathcal{SD}_{N-1}\mathcal{+F}_{N-1}^{\ast }%
\mathcal{SF}_{N-1}  \label{def K eps}
\end{equation}%
and $\mathcal{I}_{k}\left( v_{0},v_{1},...,v_{N-1}\right) =\left(
v_{0},v_{1},...,\underset{k+1}{0},...,v_{N-1}\right) ,k=0,..,N-1$. We obtain
a new optimal control problem $\mathcal{O}_{\varepsilon }$.

The hypothesis $K>0$, implies that $\mathcal{K}_{k}+\varepsilon \mathcal{I}%
_{k}>0$, for all $k=0,..,N-1$. Thus we can apply the classical results based
on the Principle of Optimality stating that the optimal cost is a quadratic
form in the state, with the weighting matrix computable via a recursion that
involves the solution of a backward discrete-time Riccati equation.

\subsection{Backward discrete-time Riccati equation of control}

We associate with $\mathcal{O}_{\varepsilon }$ the backward discrete-time
Riccati equation 
\begin{gather}
R_{n}^{\varepsilon }=\mathcal{A}^{\ast }R_{n+1}^{\varepsilon }\mathcal{A}+%
\mathcal{B}R_{n+1}^{\varepsilon }\mathcal{B}+\mathcal{C}^{\ast }\mathcal{C}%
-\left( \mathcal{D}_{n}^{\ast }R_{n+1}^{\varepsilon }\mathcal{A}+\mathcal{F}%
_{n}^{\ast }R_{n+1}^{\varepsilon }\mathcal{B}\right) ^{\ast }\cdot
\label{R(M,n)} \\
(\mathcal{K}_{n}+\varepsilon \mathcal{I}_{n}+\mathcal{D}_{n}^{\ast
}R_{n+1}^{\varepsilon }\mathcal{D}_{n}+\mathcal{F}_{n}^{\ast
}R_{n+1}^{\varepsilon }\mathcal{F}_{n})^{-1}\left( \mathcal{D}_{n}^{\ast
}R_{n+1}^{\varepsilon }\mathcal{A}+\mathcal{F}_{n}^{\ast
}R_{n+1}^{\varepsilon }\mathcal{B}\right) ,  \notag \\
\text{for~}n<N-1\text{ and}  \notag \\
R_{N-1}^{\varepsilon }=\mathcal{C}^{\ast }\mathcal{C}+\mathcal{A}^{\ast }%
\mathcal{SA+B}^{\ast }\mathcal{SB}-  \label{ci Ric} \\
\left( \mathcal{D}_{N-1}^{\ast }\mathcal{SA+F}_{N-1}^{\ast }\mathcal{SB}%
\right) ^{\ast }\cdot \left( \mathbb{K}_{N-1}^{\varepsilon }\right)
^{-1}\left( \mathcal{D}_{N-1}^{\ast }\mathcal{SA+F}_{N-1}^{\ast }\mathcal{SB}%
\right) .  \notag
\end{gather}

Setting 
\begin{equation}
W_{N-1}=-\left( \mathbb{K}_{N-1}^{\varepsilon }\right) ^{-1}\left( \mathcal{D%
}_{N-1}^{\ast }\mathcal{SA+F}_{N-1}^{\ast }\mathcal{SB}\right) ,
\label{def WN-1}
\end{equation}%
we observe that 
\begin{gather*}
R_{N-1}^{\varepsilon }=\mathcal{C}^{\ast }\mathcal{C}+\mathcal{A}^{\ast }%
\mathcal{SA+B}^{\ast }\mathcal{SB+}\left( \mathcal{D}_{N-1}^{\ast }\mathcal{%
SA+F}_{N-1}^{\ast }\mathcal{SB}\right) ^{\ast }W_{N-1}+ \\
W_{N-1}^{\ast }\left( \mathcal{D}_{N-1}^{\ast }\mathcal{SA+F}_{N-1}^{\ast }%
\mathcal{SB}\right) + \\
W_{N-1}^{\ast }\left( \mathcal{K}_{N-1}+\varepsilon \mathcal{I}_{N-1}+%
\mathcal{D}_{N-1}^{\ast }\mathcal{SD}_{N-1}\mathcal{+F}_{N-1}^{\ast }%
\mathcal{SF}_{N-1}\right) W_{N-1} \\
=\mathcal{C}^{\ast }\mathcal{C}+\left( \mathcal{A}+\mathcal{D}%
_{N-1}W_{N-1}\right) ^{\ast }\mathcal{S}\left( \mathcal{A}+\mathcal{D}%
_{N-1}W_{N-1}\right) + \\
\left( \mathcal{B}+\mathcal{F}_{N-1}W_{N-1}\right) ^{\ast }\mathcal{S}\left( 
\mathcal{B}+\mathcal{F}_{N-1}W_{N-1}\right) +W_{N-1}^{\ast }\left( \mathcal{K%
}_{N-1}+\varepsilon \mathcal{I}_{N-1}\right) W_{N-1}.
\end{gather*}%
Now it is clear that $R_{N-1}^{\varepsilon }\geq 0$. Denoting%
\begin{gather}
W_{n}=-(\mathcal{K}_{n}+\mathcal{D}_{n}^{\ast }R_{n+1}^{\varepsilon }%
\mathcal{D}_{n}+\mathcal{F}_{n}^{\ast }R_{n+1}^{\varepsilon }\mathcal{F}%
_{n}+\varepsilon I_{n})^{-1}\cdot  \label{def Wn} \\
\left( \mathcal{D}_{n}^{\ast }R_{n+1}^{\varepsilon }\mathcal{A}+\mathcal{F}%
_{n}^{\ast }R_{n+1}^{\varepsilon }\mathcal{B}\right) ,n=0,..,N-2,  \notag
\end{gather}%
and applying formula (4.8) from \cite{Ocam}, we obtain 
\begin{gather*}
R_{n}^{\varepsilon }=\left( \mathcal{A+D}_{n}W_{n}\right) ^{\ast
}R_{n+1}^{\varepsilon }\left( \mathcal{A+D}_{n}W_{n}\right) +\left( B+%
\mathcal{F}_{n}W_{n}\right) ^{\ast }R_{n+1}^{\varepsilon }\left( B+\mathcal{F%
}_{n}W_{n}\right) \\
+\mathcal{C}^{\ast }\mathcal{C}+W_{n}^{\ast }\left( \mathcal{K}%
_{n}+\varepsilon I_{n}\right) W_{n},n=0,..,N-2.
\end{gather*}%
Using the induction, we deduce that Riccati equation (\ref{R(M,n)}) has a
unique nonnegative solution $R_{n}^{\varepsilon },n=0,..,N-1$.

\begin{lemma}
\label{lem cost} The cost functional (\ref{cost eps}) can be equivalently
rewritten as 
\begin{gather}
I_{x_{0},N,\varepsilon }(U)=E\left[ \left\langle R_{0}^{\varepsilon
}X_{0},X_{0}\right\rangle \right] +  \label{form Cost ri} \\
\left\langle \mathbb{K}_{N-1}^{\varepsilon }\left(
W_{N-1}X_{N-1}-U_{N-1}\right) ,\left( W_{N-1}X_{N-1}-U_{N-1}\right)
\right\rangle + \\
\sum\limits_{k=0}^{N-2}E\left[ \left\Vert (\mathcal{K}_{n}+\mathcal{D}%
_{n}^{\ast }R_{n+1}^{\varepsilon }\mathcal{D}_{n}+\mathcal{F}_{n}^{\ast
}R_{n+1}^{\varepsilon }\mathcal{F}_{n}+\varepsilon \mathcal{I}%
_{n})^{1/2}\left( W_{n}X_{n}-U_{n}\right) \right\Vert ^{2}\right] ,  \notag
\end{gather}%
where $R_{n}^{\varepsilon }$ is the unique solution of (\ref{R(M,n)})-(\ref%
{ci Ric}).
\end{lemma}

\begin{proof}
Let $X_{n+1}$ be defined by (\ref{sl 1}). We have%
\begin{gather*}
E\left[ \left\langle R_{n+1}^{\varepsilon }X_{n+1},X_{n+1}\right\rangle %
\right] =E[\left\langle \left( \mathcal{A}^{\ast }R_{n+1}^{\varepsilon }%
\mathcal{A}+\mathcal{B}R_{n+1}^{\varepsilon }\mathcal{B}\right)
X_{n},X_{n}\right\rangle + \\
2\left\langle \left( \mathcal{D}_{n}^{\ast }R_{n+1}^{\varepsilon }\mathcal{A}%
+\mathcal{F}_{n}^{\ast }R_{n+1}^{\varepsilon }\mathcal{B}\right)
X_{n},U_{n}\right\rangle \\
+\left\langle \left( \mathcal{D}_{n}^{\ast }R_{n+1}^{\varepsilon }\mathcal{D}%
_{n}+\mathcal{F}_{n}^{\ast }R_{n+1}^{\varepsilon }\mathcal{F}_{n}\right)
U_{n},U_{n}\right\rangle ]
\end{gather*}%
and, taking into account (\ref{R(M,n)}) and (\ref{def Wn}), we obtain%
\begin{gather*}
E\left[ \left\langle R_{n+1}^{\varepsilon }X_{n+1},X_{n+1}\right\rangle %
\right] =E[\left\langle R_{n}^{\varepsilon }X_{n},X_{n}\right\rangle
-\left\langle \mathcal{C}^{\ast }\mathcal{C}X_{n},X_{n}\right\rangle - \\
\left\langle (\mathcal{K}_{n}+\varepsilon \mathcal{I}_{n})U_{n},U_{n}\right%
\rangle ]- \\
+E[\left\langle (\mathcal{K}_{n}+\mathcal{D}_{n}^{\ast }R_{n+1}^{\varepsilon
}\mathcal{D}_{n}+\mathcal{F}_{n}^{\ast }R_{n+1}^{\varepsilon }\mathcal{F}%
_{n}+\varepsilon \mathcal{I}_{n})W_{n}X_{n},W_{n}X_{n}\right\rangle \\
-2E\left[ \left\langle (\mathcal{K}_{n}+\mathcal{D}_{n}^{\ast
}R_{n+1}^{\varepsilon }\mathcal{D}_{n}+\mathcal{F}_{n}^{\ast
}R_{n+1}^{\varepsilon }\mathcal{F}_{n}+\varepsilon \mathcal{I}%
_{n})W_{n}X_{n},U_{n}\right\rangle \right] \\
+E\left[ \left\langle (\mathcal{K}_{n}+\mathcal{D}_{n}^{\ast
}R_{n+1}^{\varepsilon }\mathcal{D}_{n}+\mathcal{F}_{n}^{\ast
}R_{n+1}^{\varepsilon }\mathcal{F}_{n}+\varepsilon \mathcal{I}%
_{n})U_{n},U_{n}\right\rangle \right] \\
=E\left[ \left\langle R_{n}^{\varepsilon }X_{n},X_{n}\right\rangle
-\left\langle \mathcal{C}^{\ast }\mathcal{C}X_{n},X_{n}\right\rangle
-\left\langle (\mathcal{K}_{n}+\varepsilon \mathcal{I}_{n})U_{n},U_{n}\right%
\rangle \right] \\
+E\left\langle (\mathcal{K}_{n}+\mathcal{D}_{n}^{\ast }R_{n+1}^{\varepsilon }%
\mathcal{D}_{n}+\mathcal{F}_{n}^{\ast }R_{n+1}^{\varepsilon }\mathcal{F}%
_{n}+\varepsilon \mathcal{I}_{n})\left( W_{n}X_{n}-U_{n}\right) ,\left(
W_{n}X_{n}-U_{n}\right) \right\rangle .
\end{gather*}%
for all $n=0,..,N-2$. Summing for $n=0$ to $N-2$ the last equality, we obtain%
\begin{gather}
E\left[ \left\langle R_{N-1}^{\varepsilon }X_{N-1},X_{N-1}\right\rangle %
\right] =E\left[ \left\langle R_{0}^{\varepsilon }X_{0},X_{0}\right\rangle %
\right] -  \label{calc I} \\
\sum\limits_{k=0}^{N-2}E\left[ \left\langle \mathcal{C}^{\ast }\mathcal{C}%
X_{k},X_{k}\right\rangle \right] +E\left[ \left\langle \left( \mathcal{K}%
_{k}+\varepsilon \mathcal{I}_{k}\right) U_{k},U_{k}\right\rangle \right] 
\notag \\
+\sum\limits_{k=0}^{N-2}E\left[ \left\Vert (\mathcal{K}_{n}+\mathcal{D}%
_{n}^{\ast }R_{n+1}^{\varepsilon }\mathcal{D}_{n}+\mathcal{F}_{n}^{\ast
}R_{n+1}^{\varepsilon }\mathcal{F}_{n}+\varepsilon \mathcal{I}%
_{n})^{1/2}\left( W_{n}X_{n}-U_{n}\right) \right\Vert ^{2}\right] .  \notag
\end{gather}%
On the other hand, from (\ref{ci Ric}), (\ref{def K eps}) and (\ref{def WN-1}%
), we know that 
\begin{gather*}
E\left[ \left\langle R_{N-1}^{\varepsilon }X_{N-1},X_{N-1}\right\rangle %
\right] =E[\left\langle \left( \mathcal{C}^{\ast }\mathcal{C}+\mathcal{A}%
^{\ast }\mathcal{SA+B}^{\ast }\mathcal{SB}\right)
X_{N-1},X_{N-1}\right\rangle \\
-\left\langle \mathbb{K}_{N-1}^{\varepsilon
}W_{N-1}X_{N-1},W_{N-1}X_{N-1}\right\rangle ] \\
=E[\left\langle \left( \mathcal{C}^{\ast }\mathcal{C}+\mathcal{A}^{\ast }%
\mathcal{SA+B}^{\ast }\mathcal{SB}\right) X_{N-1},X_{N-1}\right\rangle \\
-\left\langle \mathbb{K}_{N-1}^{\varepsilon }\left(
W_{N-1}X_{N-1}-U_{N-1}\right) ,\left( W_{N-1}X_{N-1}-U_{N-1}\right)
\right\rangle \\
+2\left\langle \left( \mathcal{D}_{N-1}^{\ast }\mathcal{SA+F}_{N-1}^{\ast }%
\mathcal{SB}\right) X_{N-1},U_{N-1}\right\rangle +\left\langle \mathbb{K}%
_{N-1}^{\varepsilon }U_{N-1},U_{N-1}\right\rangle ].
\end{gather*}%
Replacing the above formula in (\ref{calc I}), we obtain (\ref{form Cost ri}%
) and the conclusion follows.
\end{proof}

\subsection{Main results}

In this section we shall prove that problem $\mathcal{O}$ has a solution
derived from the solution of problem $\mathcal{O}_{\varepsilon }$.

\begin{proposition}
\label{minimul nu dep de eps} For all $\varepsilon >0$%
\begin{equation*}
\min_{U\in \mathbb{U}_{0,N-1}^{a}}I_{x_{0},N,\varepsilon }(U)=\min_{u\in 
\mathcal{U}_{0,N-1}^{a}}I_{x_{0},N}(u).
\end{equation*}
\end{proposition}

\begin{proof}
Let $u=\{u_{0},..,u_{N-1}\}\in \mathcal{U}_{0,N-1}^{a}$. If $%
U=\{U_{0},..,U_{N-1}\},$

$U_{k}=\left( 0,..,\underset{k}{u_{k}},...0\right) $, then $U\in \mathbb{U}%
_{0,N-1}^{a}$ and $I_{x_{0},N,\varepsilon }(U)=I_{x_{0},N}(u)$. Thus 
\begin{equation}
\min_{u\in \mathcal{U}_{0,N-1}^{a}}I_{x_{0},N}(u)\geq \min_{U\in \mathbb{U}%
}I_{x_{0},N,\varepsilon }(U).  \label{prima}
\end{equation}%
On the other hand if $U\in \mathbb{U}_{0,N-1}^{a}$ and $u=\{u_{0},..,u_{N-1}%
\}$ is defined by $u_{k}=U_{kk},k=0,1,..,N-1,$ then $u\in \mathcal{U}%
_{0,N-1}^{a}$ and $I_{x_{0},N}(u)=I_{x_{0},N}(U)\leq I_{x_{0},N,\varepsilon
}(U)$. Replacing $U$ in the above inequality by $\widetilde{U}^{\varepsilon
} $ the optimal control which minimizes $I_{x_{0},N,\varepsilon }(U)$ (we
know that it exists), we see that $I_{x_{0},N}(\widetilde{u})\leq
I_{x_{0},N,\varepsilon }(\widetilde{U}^{\varepsilon })=\min_{U\in \mathbb{U}%
_{0,N-1}^{a}}I_{x_{0},N,\varepsilon }(U)$, where $\widetilde{u}=\{\widetilde{%
u}_{0},..,\widetilde{u}_{N-1}\}$ and $\widetilde{u}_{k}=\widetilde{U}_{kk}$.
Therefore $\min_{u\in \mathcal{U}_{0,N-1}^{a}}I_{x_{0},N}(u)\leq \min_{U\in 
\mathbb{U}_{0,N-1}^{a}}I_{x_{0},N,\varepsilon }(U)$. In view of (\ref{prima}%
) we get the conclusion.
\end{proof}

The next theorem is a direct consequence of Lemma \ref{lem cost} and of the
above proposition.

\begin{theorem}
\label{t opt}Let $\{R_{n}^{\varepsilon }\}_{n=0,..,N-1}$ be the unique
solution of the Riccati equation (\ref{R(M,n)})-(\ref{ci Ric}) and let $%
W_{n},n=0,..,N-1$ be defined by (\ref{def Wn}), (\ref{def WN-1}). The
control sequence $\widetilde{U}=\{\widetilde{U}_{0}=W_{0}X_{0},...,%
\widetilde{U}_{n}=W_{n}X_{n},...,\widetilde{U}_{N-1}=W_{N-1}X_{N-1}\}$
minimizes the cost functional $I_{x_{0},N,\varepsilon }(U)$ and $\min_{U\in 
\mathbb{U}_{0,N-1}^{a}}I_{x_{0},N,\varepsilon }(U)=E\left[ \left\langle
R_{0}^{\varepsilon }X_{0},X_{0}\right\rangle \right] $.

Moreover, $\{R_{n}^{\varepsilon }\}_{n=0,..,N-1}$ does not depend on $%
\varepsilon $ and the control

$\widetilde{u}=\{\widetilde{u}_{0},...,\widetilde{u}_{N-1}\}$, defined by $%
\widetilde{u}_{k}=\widetilde{U}_{kk},k=0,1,..,N-1$ is also optimal for $%
I_{x_{0},N}(u)$. The optimal cost is 
\begin{equation*}
\min_{u\in \mathcal{U}_{0,N-1}^{a}}I_{x_{0},N}(u)=E\left[ \left\langle
R_{0}^{\varepsilon }X_{0},X_{0}\right\rangle \right] .
\end{equation*}
\end{theorem}

\begin{proof}
The proof is a simple exercise for the reader.
\end{proof}

The following numerical example illustrates the applicability of the theory.

\begin{example}
\label{exp 1} Let $\alpha =\frac{1}{2},h=1,d=2,m=1,x_{0}=\left( 
\begin{array}{c}
0.2 \\ 
0.3%
\end{array}%
\right) $ and $\mathbb{A=}\left( 
\begin{array}{cc}
1 & 0 \\ 
1 & 0%
\end{array}%
\right) $, $\mathbb{B=}\left( 
\begin{array}{cc}
1 & 2 \\ 
0 & 1%
\end{array}%
\right) ,\mathbb{D}=\left( \ 
\begin{array}{c}
1 \\ 
-1%
\end{array}%
\right) ,\mathbb{F}=\left( \ 
\begin{array}{c}
2 \\ 
1%
\end{array}%
\right) ,C=\left( \ 
\begin{array}{cc}
2 & -1%
\end{array}%
\right) ,K=1,S=\left( 
\begin{array}{cc}
2 & 0 \\ 
0 & 2%
\end{array}%
\right) .$ Then $A_{0}=\left( 
\begin{array}{cc}
3/2 & 0 \\ 
1 & 1/2%
\end{array}%
\right) $ and $T=\mathbb{T}$ for $\mathbb{T}=\mathbb{B}$, $\mathbb{D}$, $%
\mathbb{F}$. We consider the optimal control problem $\mathcal{O}$ for $N=4$%
. Using a computer program, we compute in four simple steps the solution $%
R_{0}^{\varepsilon }$ of the Riccati equation (\ref{R(M,n)})-(\ref{ci Ric}).
The first five lines and columns of the matrix that defines the operator $%
R_{0}^{\varepsilon }$ are the following%
\begin{equation*}
\left( 
\begin{array}{cccccc}
266.5781 & 33.3776 & 14.1358 & 16.6723 & 7.5490 & . \\ 
33.3776 & 149.0853 & -9.7899 & 8.6506 & -4.8235 & . \\ 
14.1358 & -9.7899 & 1.7822 & 0.4176 & 0.9168 & . \\ 
16.6723 & 8.6506 & 0.4176 & 1.3910 & 0.2413 & . \\ 
7.5490 & -4.8235 & 0.9168 & 0.2413 & 0.4730 & . \\ 
. & . & . & . & . & .%
\end{array}%
\right) 
\end{equation*}
\end{example}

Using the intermediate values $R_{n}^{\varepsilon }$, (\ref{def Wn}) and
Theorem \ref{t opt}, we obtain the following optimal control sequence 
\begin{eqnarray}
\widetilde{u}_{3} &=&-0.3333x_{31}-0.6000x_{32}-0.0167x_{21}+0.0167x_{22}-
\label{U3} \\
&&0.0083x_{11}+0.0083x_{12}-0.0052x_{01}+0.0052x_{00}  \notag \\
\widetilde{u}_{2} &=&-0.4788x_{21}-0.6894x_{22}-0.0234x_{11}+0.0109x_{12}-
\label{U2} \\
&&0.0121x_{01}+0.0054x_{02}  \notag \\
\widetilde{u}_{1} &=&-0.4563x_{11}-0.7023x_{12}-0.0253x_{01}+0.0107x_{02}
\label{U1} \\
\widetilde{u}_{0} &=&-0.4582x_{01}-0.7484x_{02},  \label{U0}
\end{eqnarray}%
where $x_{n}=\left( 
\begin{array}{c}
x_{n1} \\ 
x_{n2}%
\end{array}%
\right) ,n=1,2,3$ is the state vector of the fractional system. The optimal
cost is 
\begin{equation*}
\min_{u\in \mathcal{U}_{0,N-1}^{a}}I_{x_{0},N}(u)=x_{0}^{T}\left( 
\begin{array}{cc}
266.5781 & 33.3776 \\ 
33.3776 & 149.0853%
\end{array}%
\right) x_{0}=28.086.
\end{equation*}

\section{A dynamic programming approach for the \protect\linebreak %
fractional system}

In this section we apply the Principle of Optimality to derive a direct
algorithm for solving the optimal control problem $\mathcal{O}$. As in \cite%
{polonezi}, the optimal control is a state feedback law, computable via a
recursion commencing at the terminal time and evolving backwards. The
obtained result is a stochastic counterpart of the one provided in \cite%
{polonezi}\ for deterministic fractional systems.

\begin{center}
{\LARGE Algorithm A}
\end{center}

Consider the optimal control problem $\mathcal{O}$. An optimal control
process $P_{0}$ is defined by the control policy $u=\{u_{k}\}_{k\in
\{0,..,N-1\}}$ and the corresponding trajectory $x=\{x_{k}\}_{k\in
\{0,..,N\}}.$ Let 
\begin{eqnarray*}
P_{m} &:&x_{m},x_{m+1}..,x_{N-1},x_{N} \\
&&u_{m},..,u_{N-2},u_{N-1}.
\end{eqnarray*}%
be a final segment of $P_{0}$ starting at a time $t=m,$ when system (\ref{ec
c2}) is in the state $x_{m}$ obtained from the initial state $x_{0}$ with
the optimal control sequence $u_{0},..,u_{m-1}$. The performance functional
on this final segment is 
\begin{gather}
I_{m,x_{0},..,x_{m}}(u)=  \label{cost K} \\
\sum\limits_{n=m}^{N-1}E\left[ \left( \left\Vert Cx_{n}\right\Vert
^{2}+<Ku_{n},u_{n}>\right) \right] +E<Sx_{N},x_{N}>.  \notag
\end{gather}%
The Principle of Optimality says that any final segment $P_{m}$ of $P_{0}$
must be optimal for $I_{m,x_{0},..,x_{m}}$.

Thus for $m=N-1$, the process 
\begin{eqnarray*}
P_{N-1} &:&x_{N-1},x_{N} \\
&&u_{N-1}
\end{eqnarray*}%
should be optimal for the cost 
\begin{eqnarray}
I_{N-1,x_{0},..,x_{N-1}}(u) &=&E\left[ \left( \left\Vert Cx_{N-1}\right\Vert
^{2}+\left\langle Ku_{N-1},u_{N-1}\right\rangle \right) \right] +
\label{form I N-1} \\
&&E\left[ \left\langle Sx_{N},x_{N}\right\rangle \right] .  \notag
\end{eqnarray}%
This condition and the following computations leads to a formula for the
optimal control $u_{N-1}$. Writing (\ref{ec c2}) for $k=N-1$, we obtain $%
x_{N}$. Substituting $x_{N}$ in (\ref{form I N-1}), we get 
\begin{gather}
I_{N-1,x_{0},..,x_{N-1}}(u)=  \label{form cost} \\
=E<S(\dsum\limits_{j=0}^{N-1}A_{j}x_{N-1-j}+\xi _{N-1}Bx_{N-1}+Du_{N-1}+\xi
_{N-1}Fu_{N-1}),  \notag \\
\dsum\limits_{j=0}^{N-1}A_{j}x_{N-1-j}+\xi _{N-1}Bx_{N-1}+Du_{N-1}+\xi
_{N-1}Fu_{N-1}>+  \notag \\
+E\left[ \left( \left\Vert Cx_{N-1}\right\Vert
^{2}+<Ku_{N-1},u_{N-1}>\right) \right]  \notag
\end{gather}

Since $x_{n},u_{n}$ are $\mathcal{G}_{n}$-measurable and $\xi _{p}$%
-independent for all $p\geq n,n,p\in 
%TCIMACRO{\U{2115} }%
%BeginExpansion
\mathbb{N}
%EndExpansion
$, we have 
\begin{eqnarray*}
E\left[ \left\langle Tx_{i},\xi _{n}Sv_{n},\right\rangle \right] &=&E\left[
\xi _{n}\right] E\left[ \left\langle Tx_{i},Sv_{n},\right\rangle \right] =0,
\\
E\left[ \left\langle \xi _{n}Tv_{n},\xi _{n}Sv_{n},\right\rangle \right] &=&E%
\left[ \xi _{n}^{2}\right] E\left[ \left\langle Tv_{n},Sv_{n},\right\rangle %
\right] =E\left[ \left\langle Tv_{n},Sv_{n},\right\rangle \right]
\end{eqnarray*}%
for all $i\leq n$, $v=u,x$ and $S,T$ matrices of appropriate dimensions.
Therefore,%
\begin{gather*}
I_{N-1,x_{0},..,x_{N-1}}(u)=E\left\langle
S\dsum\limits_{j=0}^{N-1}A_{j}x_{N-1-j},\dsum%
\limits_{j=0}^{N-1}A_{j}x_{N-1-j}\right\rangle + \\
E\left\langle SBx_{N-1},Bx_{N-1}\right\rangle +2E\left\langle D^{\ast
}S\dsum\limits_{j=0}^{N-1}A_{j}x_{N-1-j},u_{N-1}\right\rangle + \\
E\left\langle F^{\ast }SFu_{N-1},u_{N-1}\right\rangle +2E\left\langle
F^{\ast }SBx_{N-1},u_{N-1}\right\rangle \\
+E\left\langle D^{\ast }SDu_{N-1},u_{N-1}\right\rangle +E\left[ \left(
\left\Vert Cx_{N-1}\right\Vert ^{2}+\left\langle
Ku_{N-1},u_{N-1}\right\rangle \right) \right] \\
=E\left[ \left\Vert \sqrt{S}\dsum\limits_{j=0}^{N-1}A_{j}x_{N-1-j}\right%
\Vert ^{2}\right] +E\left[ \left\Vert \sqrt{S}Bx_{N-1}\right\Vert ^{2}\right]
+E\left[ \left\Vert Cx_{N-1}\right\Vert ^{2}\right] \\
+2E\left\langle D^{\ast }S\dsum\limits_{j=1}^{N-1}A_{j}x_{N-1-j}+\left(
F^{\ast }SB+D^{\ast }SA_{0}\right) x_{N-1},u_{N-1}\right\rangle \\
+E\left\langle \left( F^{\ast }SF+D^{\ast }SD+K\right)
u_{N-1},u_{N-1}\right\rangle .
\end{gather*}%
Setting 
\begin{gather}
v_{N-1}\left( x_{0},..,x_{N-1}\right) =D^{\ast
}S\dsum\limits_{j=1}^{N-1}A_{j}x_{N-1-j}+\left( F^{\ast }SB+D^{\ast
}SA_{0}\right) x_{N-1},  \notag \\
w_{N-1}\left( x_{0},..,x_{N-1}\right) =E\left[ \left\Vert \sqrt{S}%
\dsum\limits_{j=0}^{N-1}A_{j}x_{N-1-j}\right\Vert ^{2}\right]  \label{coef}
\\
+E\left[ \left\Vert \sqrt{S}Bx_{N-1}\right\Vert ^{2}\right] +E\left[
\left\Vert Cx_{N-1}\right\Vert ^{2}\right]  \notag \\
J_{N-1}=F^{\ast }SF+D^{\ast }SD+K>0  \notag
\end{gather}

and using a squares completion technique, we see that%
\begin{gather*}
I_{N-1,x_{0},..,x_{N-1}}(u)=w_{N-1}\left( x_{0},..,x_{N-1}\right) + \\
2E\left\langle v_{N-1}\left( x_{0},..,x_{N-1}\right) ,u_{N-1}\right\rangle
+E\left\langle J_{N-1}u_{N-1},u_{N-1}\right\rangle .
\end{gather*}%
The cost functional $I_{N-1,x_{0},..,x_{N-1}}(u)$ can be equivalently
rewritten as%
\begin{eqnarray*}
I_{N-1,x_{0},..,x_{N-1}}(u) &=&E\left\langle J_{N-1}\left(
u_{N-1}+J_{N-1}^{-1}v_{N-1}\right) ,\left(
u_{N-1}+J_{N-1}^{-1}v_{N-1}\right) \right\rangle \\
&&+w_{N-1}-E\left\langle J_{N-1}^{-1}v_{N-1},v_{N-1}\right\rangle .
\end{eqnarray*}%
As a function of $u_{N-1}$, $I_{N-1,x_{0},..,x_{N-1}}(u)$ is optimal for 
\begin{equation}
u_{N-1}^{\ast }=-J_{N-1}^{-1}v_{N-1}  \label{contr optim}
\end{equation}%
and its optimal value is 
\begin{equation}
\min_{u_{N-1}\in \mathcal{U}%
_{N-1,N-1}^{a}}I_{N-1,x_{0},..,x_{N-1}}(u)=w_{N-1}-E\left\langle
J_{N-1}^{-1}v_{N-1},v_{N-1}\right\rangle .  \label{cost optim}
\end{equation}

In view of (\ref{coef}), 
\begin{equation}
u_{N-1}^{\ast }\left( x_{0},..,x_{N-1}\right)
=\dsum\limits_{j=0}^{N-1}W_{j,N-1}x_{N-1-j},  \label{form U N-1}
\end{equation}%
where 
\begin{eqnarray}
W_{j,N-1} &=&-\left( F^{\ast }SF+D^{\ast }SD+K\right) ^{-1}D^{\ast }SA_{j}%
\text{, }j\in \{1,..,N-1\}  \label{rec -1} \\
W_{0,N-1} &=&-\left( F^{\ast }SF+D^{\ast }SD+K\right) ^{-1}\left( F^{\ast
}SB+D^{\ast }SA_{0}\right) .  \notag
\end{eqnarray}%
From the above proof we deduce that $u_{N-1}^{\ast }\left(
x_{0},..,x_{N-1}\right) =\dsum\limits_{j=0}^{N-1}W_{j,N-1}x_{N-1-j}$ is
optimal for $I_{N-1}(u)$ for any trajectory $\left( x_{0},..,x_{N-1}\right) $%
. Substituting (\ref{contr optim}) to (\ref{form cost}) we obtain%
\begin{gather*}
\min_{u_{N-1}\in \mathcal{U}_{N-1,N-1}^{a}}I_{N-1,x_{0},..,x_{N-1}}(u)=%
\mathcal{O}(x_{0},..,x_{N-1})= \\
E<S(\dsum\limits_{j=0}^{N-1}A_{j}x_{N-1-j}+\xi _{N-1}Bx_{N-1}+\left( D+\xi
_{N-1}F\right) \dsum\limits_{j=0}^{N-1}W_{j,N-1}x_{N-1-j}), \\
\dsum\limits_{j=0}^{N-1}A_{j}x_{N-1-j}+\xi _{N-1}Bx_{N-1}+\left( D+\xi
_{N-1}F\right) \dsum\limits_{j=0}^{N-1}W_{j,N-1}x_{N-1-j})>+ \\
+E\left[ \left\Vert Cx_{N-1}\right\Vert ^{2}\right] +E[<K\dsum%
\limits_{j=0}^{N-1}W_{j,N-1}x_{N-1-j},\dsum%
\limits_{j=0}^{N-1}W_{j,N-1}x_{N-1-j}>] \\
=E<S(\dsum\limits_{j=0}^{N-1}V_{N-1,j}^{S,1}x_{N-1-j}+\xi
_{N-1}\dsum\limits_{j=0}^{N-1}V_{N-1,j}^{S,2}x_{N-1-j}), \\
\dsum\limits_{j=0}^{N-1}V_{N-1,j}^{S,1}x_{N-1-j}+\xi
_{N-1}\dsum\limits_{j=0}^{N-1}V_{N-1,j}^{S,2}x_{N-1-j}>+ \\
+E\left[ \left\Vert Cx_{N-1}\right\Vert ^{2}\right] +E[<K\dsum%
\limits_{j=0}^{N-1}W_{j,N-1}x_{N-1-j},\dsum%
\limits_{j=0}^{N-1}W_{j,N-1}x_{N-1-j}>]
\end{gather*}

Denoting 
\begin{eqnarray}
V_{N-1,0}^{S,2} &=&B+FW_{0,N-1},V_{N-1,j}^{S,2}=FW_{j,N-1},j\neq 0,
\label{rec 0} \\
V_{N-1,j}^{S,1} &=&A_{j}+DW_{j,N-1},V_{N-1,j}^{K,1}=W_{j,N-1},j\in
\{0,..,N-1\}  \notag
\end{eqnarray}%
we obtain the optimal value of the cost: 
\begin{gather*}
\mathcal{O}(x_{0},..,x_{N-1})=E\left[ \left\Vert \sqrt{S}\dsum%
\limits_{j=0}^{N-1}V_{N-1,j}^{S,1}x_{N-1-j}\right\Vert ^{2}\right] + \\
E\left[ \left\Vert \sqrt{S}\dsum\limits_{j=0}^{N-1}V_{N-1,j}^{S,2}x_{N-1-j}%
\right\Vert ^{2}\right] + \\
E\left[ \left\Vert Cx_{N-1}\right\Vert ^{2}\right] +E\left[ \left\Vert \sqrt{%
K}\dsum\limits_{j=0}^{N-1}V_{N-1,j}^{K,1}x_{N-1-j}\right\Vert ^{2}\right]
\end{gather*}

Now we assume that%
\begin{eqnarray*}
P_{N-2} &:&x_{N-2},x_{N-1},x_{N} \\
&&u_{N-2},u_{N-1}
\end{eqnarray*}%
is a final segment of the process $P_{0}$.

Then $P_{N-2}$ should be optimal for $I_{N-2,x_{0},..,x_{N-2}}(u)$. Since%
\begin{gather*}
\min_{u_{N-2},u_{N-1}\in \mathcal{U}%
_{N-2,N-1}^{a}}I_{N-2,x_{0},..,x_{N-2}}(u)= \\
\min_{u_{N-2},u_{N-1}\in \mathcal{U}_{N-2,N-1}^{a}}%
\{I_{N-1,x_{0},..,x_{N-1}}\left( u\right) +E\left[ \left\Vert
Cx_{N-2}\right\Vert ^{2}\right] + \\
E\left[ \left\langle Ku_{N-2},u_{N-2}\right\rangle \right] \}= \\
\min_{u_{N-2}\in \mathcal{U}_{N-2,N-2}^{a}}\{\min_{u_{N-1}\in \mathcal{U}%
_{N-1,N-1}^{a}}I_{N-1,x_{0},..,x_{N-1}}\left( u\right) +E\left[ \left\Vert
Cx_{N-2}\right\Vert ^{2}\right] + \\
+E\left[ \left\langle Ku_{N-2},u_{N-2}\right\rangle \right] \}= \\
\min_{u_{N-2}\in \mathcal{U}_{N-2,N-2}^{a}}\{\mathcal{O}(x_{0},..,x_{N-1})+E%
\left[ \left\Vert Cx_{N-2}\right\Vert ^{2}+<Ku_{N-2},u_{N-2}>\right] \},
\end{gather*}%
it follows that $u_{N-1}$ is given by (\ref{form U N-1}) and $u_{N-2}$
should be computed. Substituting $x_{N-1}$ given by (\ref{ec c2}) in $%
\mathcal{O}(x_{0},..x_{N-2},x_{N-1}),$ we see that $\mathcal{O}%
(x_{0},..x_{N-2},x_{N-1})=\phi \left( x_{0},..x_{N-2},u_{N-2}\right) $ and $%
u_{N-2}$ solves the optimal control problem%
\begin{gather*}
\min_{u_{N-2},u_{N-1}\in \mathcal{U}%
_{N-2,N-1}^{a}}I_{N-2,x_{0},..,x_{N-2}}(u)= \\
=\min_{u_{N-2}\in \mathcal{U}_{N-2,N-2}^{a}}\{\phi \left(
x_{0},..x_{N-2},u_{N-2}\right) +E\left[ \left( \left\Vert
Cx_{N-2}\right\Vert ^{2}+<Ku_{N-2},u_{N-2}>\right) \right] \}.
\end{gather*}%
Using again the squares completion technique, we can prove that the optimal
control $u_{N-2}$ is a linear function of $x_{0},..x_{N-2}$ and $\underset{%
u_{N-2},u_{N-1}\in \mathcal{U}_{N-2,N-1}^{a}}{\min }I_{N-2}(u)$ is a
function of the trajectory $x_{0},..,x_{N-2}$. Repeating the above
arguments, we find $u_{N-3},u_{N-4}$ and so on. The general step of the
above algorithm is described in detail in the Appendix. At the step $q$ we
find the optimal control $u_{N-q}$ as a linear function of $x_{0},..,x_{N-q}$
of the form $\dsum\limits_{j=0}^{N-q}W_{j,N-q}x_{N-q-j}$ where the
coefficients $W_{j,N-q}$ are given by a set of recurrent formulas (see (\ref%
{v n-q}), (\ref{rec 1}), (\ref{rec 2}),(\ref{rec 3}) in the Appendix). At a
first sight this algorithm is more complicated than the one described in
Section 3.

\begin{example}
Consider the optimal control problem $\mathcal{O}$ under the hypotheses of
Example \ref{exp 1}. Implementing in MATLAB the Algorithm $A$, we obtain the
following results. Since $W_{0,3}=\left( 
\begin{array}{cc}
-0.3333 & -0.600%
\end{array}%
\right) ,W_{1,3}=\left( 
\begin{array}{cc}
-0.0167 & 0.0167%
\end{array}%
\right) $, $W_{2,3}=\left( 
\begin{array}{cc}
-0.0083 & 0.0083%
\end{array}%
\right) $ and\ $W_{3,3}=\left( 
\begin{array}{cc}
-0.0052 & 0.0052%
\end{array}%
\right) $ we deduce by (\ref{form U N-1}) that the optimal control $u_{3}$
have the same formula as the one obtained in Example \ref{exp 1}. Further,
we compute $V_{3,j}^{S,1},V_{3,j}^{S,2}$,$V_{3,j}^{K,1},j=0,1,2,3$. Writing (%
\ref{form R}) for $q=2$ we get $J_{2}=67.3667$. Also the coefficients of $%
x_{0},x_{1}$ and $x_{2}$ from (\ref{v n-q}) are $\left( 
\begin{array}{cc}
0.8151 & -0.3630%
\end{array}%
\right) $, $\left( 
\begin{array}{cc}
1.5750 & -0.7333%
\end{array}%
\right) $ and $\left( 
\begin{array}{cc}
32.2583 & 46.4417%
\end{array}%
\right) $, respectively. Since $W_{j,N-q}$ is obtained by multiplying the
coefficient of $x_{N-q-j}$ from $v_{N-q}$ with $-J_{N-q}^{-1}$, we get $%
W_{0,2}=-J_{2}^{-1}\left( 
\begin{array}{cc}
32.2583 & 46.4417%
\end{array}%
\right) =\left( 
\begin{array}{cc}
-0.478\,85 & -0.689\,39%
\end{array}%
\right) $,

\noindent $W_{1,2}=\left( 
\begin{array}{cc}
-0.0233 & 0.0108%
\end{array}%
\right) $ and $W_{2,2}=\left( 
\begin{array}{cc}
-0.0120 & 0.0053%
\end{array}%
\right) $. Thus 
\begin{eqnarray*}
u_{2} &=&\left( 
\begin{array}{cc}
-0.478\,85 & -0.689\,39%
\end{array}%
\right) x_{2}+\left( 
\begin{array}{cc}
-0.0233 & 0.0108%
\end{array}%
\right) x_{1}+ \\
&&\left( 
\begin{array}{cc}
-0.0120 & 0.0053%
\end{array}%
\right) x_{0}
\end{eqnarray*}%
and the formula of $\widetilde{u}_{2}$ obtained in Example \ref{exp 1} is
recovered. At the next step are computed the coefficients $%
V_{3,j}^{S,l},l=1,..,4,V_{3,j}^{S,l}$,$l=1,2,3,V_{3,j}^{K,1},l=1,2$, $j=0,1,2
$. With (\ref{form R}) written for $q=1$ we obtain $J_{1}=196.1711$. The
coefficients of $x_{0}$ and $x_{1}$ from (\ref{v n-q}) are $\left( 
\begin{array}{cc}
-0,0253 & 0.0107%
\end{array}%
\right) $, $\left( 
\begin{array}{cc}
4.8670 & -2.2318%
\end{array}%
\right) $ and 
\begin{equation*}
u_{1}=\left( 
\begin{array}{cc}
-0.456\,34 & -0.70234%
\end{array}%
\right) x_{1}+\left( 
\begin{array}{cc}
-0,0253 & 0.0107%
\end{array}%
\right) x_{0}.
\end{equation*}%
Continuing the procedure, we obtain $u_{0}=\left( 
\begin{array}{cc}
-0.4589 & -0.7481%
\end{array}%
\right) $ and $\mathcal{O}\left( x_{0}\right) =x_{0}^{T}\left( 
\begin{array}{cc}
266.8471 & 32.9452 \\ 
32.9452 & 149.4033%
\end{array}%
\right) x_{0}=28.074$.
\end{example}

\section{Conclusions}

This paper provides two methods of solving the LQ optimal control problem $%
\mathcal{O}$. Both of them are based on the dynamic programming approach.
The first one seems to be new and easier. It consists in a reformulation of
the problem for an associated linear non-fractional system (\ref{sl 1})-(\ref%
{ci sl1}), defined on spaces of higher dimensions. The second one uses the
Principle of Optimality to derive a dynamic programming algorithm for the
optimal control of the LFS. This algorithm is a stochastic counterpart of
the one obtained in \cite{polonezi} for deterministic LFSs; it keeps the
dimensions of the state space of system (\ref{ec c2})-(\ref{ci c2}), but it
is more laborious. The computer program implementing it is not such simple
and fast as the one that implements the first method. A future analysis of
these algorithms from the computer science point of view will highlight the
real advantages and disadvantages of each method.

\section{Appendix}

\begin{center}
{\LARGE The general step of the algorithm }$A$
\end{center}

Our problem is to find the final segment 
\begin{eqnarray*}
P_{N-q} &:&x_{N-q},x_{N-q+1}..,x_{N-1},x_{N} \\
&&u_{N-q},..,u_{N-2},u_{N-1}.
\end{eqnarray*}%
of $P_{0}$ which minimizes $I_{N-q,,x_{0},..,x_{N-q}}(u)$. Assume that the
optimal controls $u_{N-1},...,u_{N-q+1}$ ,$q\geq 2$ were determined and the
optimal cost 
\begin{equation*}
I_{N-q+1,,x_{0},..,x_{N-q+1}}(u)
\end{equation*}
has the form 
\begin{gather}
\mathcal{O}(x_{0},..x_{N-q},x_{N-q+1}):=\min_{u_{N-q+1},...,u_{N-1}\in 
\mathcal{U}_{N-q+1,N-1}^{a}}I_{N-q+1,x_{0},..,x_{N-q+1}}(u)=  \notag \\
\dsum\limits_{l=1}^{2^{q-1}}E\left\Vert \sqrt{S}\dsum%
\limits_{j=0}^{N-q+1}V_{N-q+1,j}^{S,l}x_{N-q+1-j}\right\Vert ^{2}+  \notag \\
\dsum\limits_{l=1}^{2^{q-1}-1}E[\left\Vert \sqrt{K}\dsum%
\limits_{j=0}^{N-q+1}V_{N-q+1,j}^{K,l}x_{N-q+1-j}\right\Vert ^{2}  \notag \\
\dsum\limits_{l=1}^{2^{q-1}-2}E[\left\Vert
C\dsum\limits_{j=0}^{N-q+1}V_{N-q+1,j}^{C,l}x_{N-q+1-j}\right\Vert ^{2}+E 
\left[ \left\Vert Cx_{N-q+1}\right\Vert ^{2}\right]  \notag \\
=\dsum\limits_{l=1}^{2^{q-1}}\sigma
_{N-q+1}^{S,l}+\dsum\limits_{l=1}^{2^{q-1}-1}\sigma
_{N-q+1}^{K,l}+\dsum\limits_{l=1}^{2^{q-1}-2}\sigma _{N-q+1}^{C,l}+E\left[
\left\Vert Cx_{N-q+1}\right\Vert ^{2}\right]  \label{form opt1}
\end{gather}%
where $V_{N-q+1,j}^{S,l},V_{N-q+1,j}^{K,l}$ and $V_{N-q+1,j}^{C,l}$ are
matrices of appropriate dimensions depending on the coefficients of the
optimal control problem. We shall compute the optimal control $u_{N-q}$ and
we shall prove that $\mathcal{O}(x_{0},..x_{N-q})$ is given by a formula of
the form (\ref{form opt1}) where $q$ is replaced by $q-1$.

We know that 
\begin{equation*}
I_{N-q,,x_{0},..,x_{N-q}}(u)=I_{N-q+1,x_{0},..,x_{N-q+1}}(u)+E\left[ \left(
\left\Vert Cx_{N-q}\right\Vert ^{2}+<Ku_{N-q},u_{N-q}>\right) \right] .
\end{equation*}%
Then 
\begin{gather*}
\min_{u_{N-q},...,u_{N-1}\in \mathcal{U}%
_{N-q,N-1}^{a}}I_{N-q,x_{0},..,x_{N-q+1}}(u)= \\
\min_{u_{N-q}\in \mathcal{U}_{N-q,N-q}^{a}}\{\min_{u_{N-q+1},...,u_{N-1}\in 
\mathcal{U}_{N-q+1,N-1}^{a}}I_{x,N-q+1}(u)+E\left[ \left( \left\Vert
Cx_{N-q}\right\Vert ^{2}+\left\langle Ku_{N-q},u_{N-q}\right\rangle \right) %
\right] \} \\
:=\min_{u_{N-q}\in \mathcal{U}_{N-q,N-q}^{a}}f\left(
x_{0},..,x_{N-q},x_{N-q+1},u_{N-q}\right) .
\end{gather*}

Substituting 
\begin{equation}
x_{N-q+1}=\dsum\limits_{j=0}^{N-q}A_{j}x_{N-q-j}+\xi
_{N-q}Bx_{N-q}+Du_{N-q}+\xi _{N-q}Fu_{N-q}  \label{X N-q}
\end{equation}%
in $\sigma _{N-q+1}^{S,l}$ (see (\ref{form opt1})) we obtain $\sigma
_{N-q+1}^{S,l}$ as a function of the known $x_{j},j=0,N-q$ and the unknown $%
u_{N-q}.$ We have 
\begin{gather*}
\sigma _{N-q+1}^{S,l}= \\
E\left\Vert \sqrt{S}(\dsum%
\limits_{j=1}^{N-q+1}V_{N-q+1,j}^{S,l}x_{N-q+1-j}+V_{N-q+1,0}^{S,l}\dsum%
\limits_{j=0}^{N-q}A_{j}x_{N-q-j}+V_{N-q+1,0}^{S,l}Du_{N-q})\right\Vert ^{2}
\\
+E\left\Vert \sqrt{S}V_{N-q+1,0}^{S,l}(Bx_{N-q}+Fu_{N-q})\right\Vert ^{2} \\
=E\left\Vert \sqrt{S}\dsum\limits_{j=0}^{N-q}\left(
V_{N-q+1,j+1}^{S,l}+V_{N-q+1,0}^{S,l}A_{j}\right) x_{N-q-j}\right\Vert
^{2}+E\left\Vert \sqrt{S}V_{N-q+1,0}^{S,l}Bx_{N-q}\right\Vert ^{2} \\
+2E<S\dsum\limits_{j=0}^{N-q}\left(
V_{N-q+1,j+1}^{S,l}+V_{N-q+1,0}^{S,l}A_{j}\right)
x_{N-q-j},V_{N-q+1,0}^{S,l}Du_{N-q}> \\
+2E\left\langle \left( V_{N-q+1,0}^{S,l}\right) ^{\ast
}SV_{N-q+1,0}^{S,l}Bx_{N-q},Fu_{N-q})\right\rangle \\
+E\left\Vert \sqrt{S}V_{N-q+1,0}^{S,l}Du_{N-q}\right\Vert ^{2}+E\left\Vert 
\sqrt{S}V_{N-q+1,0}^{S,l}Fu_{N-q}\right\Vert ^{2}
\end{gather*}%
A similar computation leads to a formula for $\sigma _{N-q+1}^{K}$ and $%
\sigma _{N-q+1}^{C}$ obtained from the one above by replacing $S$ by $K$. \
Also 
\begin{gather*}
E\left[ \left\Vert Cx_{N-q+1}\right\Vert ^{2}\right] =E\left\Vert
C\dsum\limits_{j=0}^{N-q}A_{j}x_{N-q-j}\right\Vert ^{2}+E\left\Vert
CBx_{N-q}\right\Vert ^{2}+ \\
2E<D^{\ast }C^{\ast
}C\dsum\limits_{j=0}^{N-1}A_{j}x_{N-q-j},u_{N-q}>+2E\left\langle F^{\ast
}C^{\ast }CBx_{N-q},u_{N-q}\right\rangle \\
+E\left\Vert CFu_{N-q}\right\Vert ^{2}+E\left\Vert CDu_{N-q}\right\Vert ^{2}.
\end{gather*}

Therefore, substituting (\ref{X N-q}) in $f\left(
x_{0},..,x_{N-q},x_{N-q+1},u_{N-q}\right) $ we see that $f\left(
x_{0},..,x_{N-q},x_{N-q+1},u_{N-q}\right) $ becomes a function $%
g(x_{0},..,x_{N-q},u_{N-q})$ of the form 
\begin{equation*}
w_{N-q}\left( x_{0},..,x_{N-q}\right) +2E\left\langle v_{N-q}\left(
x_{0},..,x_{N-q}\right) ,u_{N-q}\right\rangle +E\left\langle
J_{N-q}u_{N-q},u_{N-q}\right\rangle ,
\end{equation*}%
where 
\begin{gather}
J_{N-q}=K+\dsum\limits_{l=1}^{2^{q-1}}\left[ \left(
V_{N-q+1,0}^{S,l}D\right) ^{\ast }SV_{N-q+1,0}^{S,l}D+\left(
V_{N-q+1,0}^{S,l}F\right) ^{\ast }SV_{N-q+1,0}^{S,l}F\right]  \label{form R}
\\
+\dsum\limits_{l=1}^{2^{q-1}-1}\left[ \left( V_{N-q+1,0}^{K,l}D\right)
^{\ast }KV_{N-q+1,0}^{K,l}D+\left( V_{N-q+1,0}^{K,l}F\right) ^{\ast
}KV_{N-q+1,0}^{K,l}F\right]  \notag \\
+\dsum\limits_{l=1}^{2^{q-1}-2}\left[ \left( V_{N-q+1,0}^{C,l}D\right)
^{\ast }C^{\ast }CV_{N-q+1,0}^{C,l}D+\left( V_{N-q+1,0}^{C,l}F\right) ^{\ast
}C^{\ast }CV_{N-q+1,0}^{C,l}F\right]  \notag \\
+F^{\ast }C^{\ast }CF+D^{\ast }C^{\ast }CD>0  \notag
\end{gather}%
and 
\begin{gather}
v_{N-q}\left( x_{0},..,x_{N-q}\right) =  \label{v n-q} \\
\dsum\limits_{l=1}^{2^{q-1}}\{\left( V_{N-q+1,0}^{S,l}D\right) ^{\ast
}S\dsum\limits_{j=0}^{N-q}\left(
V_{N-q+1,j+1}^{S,l}+V_{N-q+1,0}^{S,l}A_{j}\right) x_{N-q-j}+  \notag \\
\left( V_{N-q+1,0}^{S,l}F\right) ^{\ast }SV_{N-q+1,0}^{S,l}Bx_{N-q}\}+ 
\notag \\
\dsum\limits_{l=1}^{2^{q-1}-1}\{\left( V_{N-q+1,0}^{K,l}D\right) ^{\ast
}K\dsum\limits_{j=0}^{N-q}\left(
V_{N-q+1,j+1}^{K,l}+V_{N-q+1,0}^{K,l}A_{j}\right) x_{N-q-j}+  \notag \\
\left( V_{N-q+1,0}^{K,l}F\right) ^{\ast }KV_{N-q+1,0}^{K,l}Bx_{N-q}\}+ 
\notag \\
\dsum\limits_{l=1}^{2^{q-1}-2}\{\left( V_{N-q+1,0}^{C,l}D\right) ^{\ast
}C^{\ast }C\dsum\limits_{j=0}^{N-q}\left(
V_{N-q+1,j+1}^{C,l}+V_{N-q+1,0}^{C,l}A_{j}\right) x_{N-q-j}+  \notag \\
\left( V_{N-q+1,0}^{C,l}F\right) ^{\ast }C^{\ast
}CV_{N-q+1,0}^{C,l}Bx_{N-q}\}+  \notag \\
D^{\ast }C^{\ast }C\dsum\limits_{j=0}^{N-q}A_{j}x_{N-q-j}+F^{\ast }C^{\ast
}CBx_{N-q}.  \notag
\end{gather}

Reasoning as in the case $q=1,$ we get the optimal control 
\begin{eqnarray}
u_{N-q}^{\ast }\left( x_{0},..,x_{N-q}\right) &=&-J_{N-q}^{-1}v_{N-q}
\label{form cost opt} \\
&=&\dsum\limits_{j=0}^{N-q}W_{j,N-q}x_{N-q-j}.  \notag
\end{eqnarray}%
Taking into account (\ref{v n-q}), we see that for all $j\in \{0,1,...,N-q\}$%
, $W_{j,N-q}$ is obtained by multiplying the coefficient of $x_{N-q-j}$,
from $v_{N-q},$ with $-J_{N-q}^{-1}$.

Replacing (\ref{form cost opt}) in $\sigma _{N-q+1}^{S,l}$, we observe that,
for all $l\in \{1,..,2^{q-1}\},$ 
\begin{gather*}
\sigma _{N-q+1}^{S,l}= \\
E\left\Vert S\dsum\limits_{j=0}^{N-q}\left(
V_{N-q+1,j+1}^{S,l}+V_{N-q+1,0}^{S,l}A_{j}\right)
x_{N-q-j}+V_{N-q+1,0}^{S,l}Du_{N-q}\right\Vert ^{2}+ \\
+E\left\Vert \sqrt{S}V_{N-q+1,0}^{S,l}(Bx_{N-q}+Fu_{N-q})\right\Vert ^{2} \\
=E\left\Vert \sqrt{S}\dsum\limits_{j=0}^{N-q}V_{N-q,j}^{S,l}x_{N-q-j}\right%
\Vert ^{2}+E\left\Vert \sqrt{S}\dsum%
\limits_{j=0}^{N-q}V_{N-q,j}^{S,l+2^{q-1}}x_{N-q-j}\right\Vert ^{2}
\end{gather*}%
and 
\begin{equation*}
\dsum\limits_{l=1}^{2^{q-1}}\sigma
_{N-q+1}^{S,l}=\dsum\limits_{l=1}^{2^{q}}E\left\Vert \sqrt{S}%
\dsum\limits_{j=0}^{N-q}V_{N-q,j}^{S,l}x_{N-q-j}\right\Vert ^{2},
\end{equation*}%
where%
\begin{eqnarray}
V_{N-q,j}^{S,l} &=&V_{N-q+1,j+1}^{S,l}+V_{N-q+1,0}^{S,l}\left(
A_{j}+DW_{j,N-q}\right) ,  \label{rec 1} \\
V_{N-q,0}^{S,l+2^{q-1}} &=&V_{N-q+1,0}^{S,l}\left( B+FW_{0,N-q}\right) , 
\notag \\
V_{N-q,j}^{S,l+2^{q-1}} &=&V_{N-q+1,0}^{S,l}FW_{j,N-q},j\neq 0.  \notag
\end{eqnarray}%
Arguing as above and using (\ref{cost optim}), we obtain 
\begin{gather*}
\dsum\limits_{l=1}^{2^{q-1}-1}E[\left\Vert \sqrt{K}\dsum%
\limits_{j=0}^{N-q+1}V_{N-q+1,j}^{K,l}x_{N-q+1-j}\right\Vert
^{2}+\left\langle Ku_{N-q},u_{N-q}\right\rangle = \\
=\dsum\limits_{l=1}^{2^{q}-1}E\left\Vert \sqrt{K}\dsum%
\limits_{j=0}^{N-q}V_{N-q,j}^{K,l}x_{N-q-j}\right\Vert
\end{gather*}%
where for all $l\in \{1,..,2^{q-1}-1\},j\leq N-q$ 
\begin{eqnarray}
V_{N-q,j}^{K,l} &=&V_{N-q+1,j+1}^{K,l}+V_{N-q+1,0}^{K,l}\left(
A_{j}+DW_{j,N-q}\right) ,  \label{rec 2} \\
V_{N-q,0}^{K,l+2^{q-1}-1} &=&V_{N-q+1,0}^{K,l}\left( B+FW_{0,N-q}\right) , 
\notag \\
V_{N-q,j}^{K,l+2^{q-1}-1}
&=&V_{N-q+1,0}^{K,l}FW_{j,N-q},V_{N-q,j}^{K,2^{q}-1}=W_{j,N-q},j\neq 0. 
\notag
\end{eqnarray}%
Similarly, 
\begin{gather*}
\dsum\limits_{l=1}^{2^{q-1}-2}E[\left\Vert
C\dsum\limits_{j=0}^{N-q+1}V_{N-q+1,j}^{K,l}x_{N-q+1-j}\right\Vert ^{2}+E 
\left[ \left\Vert Cx_{N-q+1}\right\Vert ^{2}\right] ^{2}= \\
\dsum\limits_{l=1}^{2^{q}-2}E\left\Vert
C\dsum\limits_{j=0}^{N-q}V_{N-q,j}^{C,l}x_{N-q-j}\right\Vert
\end{gather*}%
where $V_{N-q,j}^{C,l},1\leq l\leq 2^{q-1}-2,j\leq N-q$ are given by 
\begin{eqnarray}
V_{N-q,j}^{C,l} &=&V_{N-q+1,j+1}^{C,l}+V_{N-q+1,0}^{C,l}\left(
A_{j}+DW_{j,N-q}\right) ,  \label{rec 3} \\
V_{N-q,0}^{C,l+2^{q-1}-2} &=&V_{N-q+1,0}^{C,l}\left( B+FW_{0,N-q}\right) , 
\notag \\
V_{N-q,j}^{C,l+2^{q-1}-2}
&=&V_{N-q+1,0}^{C,l}FW_{j,N-q},V_{N-q,j}^{C,2^{q}-3}=A_{j}+DW_{j,N-q},j\neq 0
\notag \\
V_{N-q,0}^{C,2^{q}-2} &=&B+FW_{0,N-q},V_{N-q,j}^{C,2^{q}-2}=FW_{j,N-q},j\neq
0.  \notag
\end{eqnarray}

Now it is clear that a formula for $\mathcal{O}(x_{0},..,x_{N-q})$ can be
obtained by replacing $q$ with $q+1$ in (\ref{form opt1}) and using the
coefficients (\ref{rec -1}), (\ref{rec 0}), (\ref{rec 1}), (\ref{rec 2}) and
(\ref{rec 3}).

The optimal cost $I_{x_{0},N}(u)$ is given by $\mathcal{O}(x_{0})$, i.e. by
formula (\ref{form opt1}) written for $q=N+1$.

\end{document}